\documentclass[12pt]{article}
\usepackage{amsmath}
\usepackage{amsthm}
\usepackage{amsfonts}
\usepackage{geometry}
\usepackage{graphicx}
\usepackage{algorithm}
\usepackage{algpseudocode}
\usepackage{longtable}
\usepackage{color}
\usepackage{hyperref}
\usepackage{authblk}
\usepackage[backend=bibtex]{biblatex}
\newtheorem{lemma}{Lemma}
\newtheorem{theorem}{Theorem}
\newtheorem{proposition}{Proposition}

\newcommand{\C}{\mathbb{C}}
\newcommand{\R}{\mathbb{R}}
\newcommand{\Q}{\mathbb{Q}}
\newcommand{\Z}{\mathbb{Z}}
\newcommand{\F}{\mathbb{F}}
\newcommand{\Fd}{\mathbb{F}[\sqrt{d}]}
\newcommand{\Fix}{\operatorname{Fix}}

\title{Complex to Rational Fast Matrix Multiplication}

\author[1]{Yoav Moran\thanks{Email: yoav.gross@mail.huji.ac.il}}
\author[1]{Oded Schwartz\thanks{Email: odedsc@cs.huji.ac.il}}
\author[2]{Shuncheng Yuan\thanks{Email: syuan@ucsd.edu}}

\affil[1]{The Hebrew University of Jerusalem}
\affil[2]{University of California San Diego}

\date{}
\begin{document}
\maketitle

\begin{abstract}
Fast matrix multiplication algorithms are asymptotically faster than the classical cubic-time algorithm, but they are often slower in practice.
One important obstacle is the use of complex coefficients, which increases arithmetic overhead and limits practical efficiency.
This paper focuses on transforming complex-coefficient matrix multiplication schemes into equivalent real- or rational-coefficient ones. 
We present a systematic method that, given a complex-coefficient scheme, either constructs a family of equivalent rational algorithms or proves that no equivalent rational scheme exists.
Our approach relies only on basic linear-algebraic properties of similarity transformations of complex matrices.
This method recovers the previously known ad hoc results of Dumas, Pernet, and Sedoglavic (2025) and extends them to more general settings, including algorithms involving rational coefficients and square roots, with $i = \sqrt{-1}$ as a special case. 

Using this framework, we show that no rational scheme is equivalent to the Smirnov's $\langle 4,4,9,104\rangle$ $\Q[\sqrt{161}]$ algorithm (2022) and that no real scheme is equivalent to the $\langle 4,4,4,48\rangle$ complex algorithm of Kaporin (2024). 
More generally, our approach can also be used to prove the non-existence of integer-coefficient schemes.
\end{abstract}

\section{Introduction}
\subsection{Background}
Matrix multiplication is a fundamental primitive used in many applications such as numerical linear algebra, machine learning, data analysis, and graph algorithms. 
The classical algorithm multiplies two $n\times n$ matrices using $O(n^3)$ operations. 
Strassen~\cite{Strassen1969} was the first to discover a sub-cubic $O(n^{\log_2 7})$ algorithm. 
Since then, a long line of work has improved the best known upper bounds on the matrix multiplication exponent $\omega$ (e.g., \cite{Pan1978, Bini1979, Schonhage1981, Romani1982, CoppersmithWinograd1981, Strassen1986, CoppersmithWinograd1990, Stothers2010, Williams2012, LeGall2014, AlmanWilliams2024, Duan2022, Williams2024, Alman2024}), where the current best bound is $\omega<2.371339$~\cite{Alman2024}. 
However, many of these improvements come with extremely large hidden constants and are relevant only for astronomically large matrix dimensions.

As a result, research on practical matrix multiplication has diverged into several directions.
One direction seeks to improve the asymptotic bounds while retaining applicability to matrices of feasible dimension. 
These techniques include trilinear aggregation~\cite{Pan1978,PAN80,PAN82,LPS92,DIS11,HS23,SZ25}, flip-graphs-based search~\cite{KM23,KM25,AIH24,MP25}, SAT-solver-based search~\cite{CBH11,HKS19,HKS21,Yan24}, alternating least squares~\cite{JM89,m33640,BB15}, zero-padding~\cite{DIS11,SED17,FBH22}, and reinforcement learning~\cite{FBH22,SL24,NVEE25}.
Recent work also studies the numerical stability of fast matrix multiplication algorithms~\cite{Bini1980, Demmel2006, Ballard2015, Vermeylen2024, Schwartz2024}, and aims to reduce the leading coefficient~\cite{AB1, beniamini2020sparsifying, AB2, maartensson2025number}. 

Some fast matrix multiplication algorithms rely on complex coefficients, including Kaporin’s algorithms~\cite{m44448c} and schemes discovered by the Google DeepMind team using the AlphaEvolve coding agent~\cite{NVEE25}. 
Arithmetic over complex numbers incurs higher computational overhead than arithmetic over the reals, which can limit the practical efficiency of such algorithms. 
Consequently, transforming a complex-coefficient scheme into an equivalent real-coefficient scheme can lead to substantial performance improvements in practice.

Following the discovery of complex-coefficient schemes in~\cite{NVEE25}, Dumas, Pernet, and Sedoglavic applied De-Groote actions to construct equivalent rational-coefficient schemes~\cite{m34763,m44448}. 
Their approach exploits symmetry properties of the tensors together with carefully chosen, problem-specific heuristics.

\subsection{Our Contribution}
Our main contribution is a necessary and sufficient condition for a complex-valued matrix multiplication scheme to admit an equivalent rational scheme (under De Groote actions).
We further provide an explicit algorithm that deccides whether such an equivalent rational form exists and constructs one when it does.
Using this method, we recover the results of Dumas, Pernet, and Sedoglavic.

We further establish new non-existence results.
In particular, we show that Smirnov’s scheme~\cite{m494104} over $\Q(\sqrt{161})$ admits no equivalent rational scheme, and that Kaporin’s complex-coefficient scheme~\cite{m44448c} admits no equivalent real scheme. 
In particular, we show that no rational scheme is equivalent to Smirnov’s algorithm over $\Q(\sqrt{161})$, and that no real scheme is equivalent to Kaporin’s complex algorithm~\cite{m44448c}.
We also extend our framework to study whether a rational scheme can be transformed into an equivalent integer scheme.
Using this extension, we prove that the schemes produced by AlphaEvolve~\cite{NVEE25} admit no integer equivalents, and neither do several known $\mathbb{Z}[1/2]$ schemes~\cite{m33640, Hopcroft1971, Ballard2015}.

\section{Preliminary}
\subsection{Notations}
Throughout this paper, we denote $M_{m,n}(R)$ and $M_n(R)$ to be the set of $m$-by-$n$ matrices and the set of $n$-by-$n$ square matrices whose entries are in the ring $R$, respectively. $GL_n(R) = \{A\in M_n(R):A \text{ invertible}\}$ denotes the general linear group. 

Let $\Fd = \{a+b\sqrt{d}:a,b\in\F\}$ be a degree 2 extension over $\F$ whose characteristic is 0 (in most cases $\F = \Q$ or $\R$). For an element $a +b\sqrt{d}\in\Fd$, define its conjugate $\overline{a+b\sqrt{d}}=a - b\sqrt{d}$. For a matrix $A\in M_{m,n}(\Fd)$, let $\overline{A}$ denote the entry-wise conjugate of $A$, and $A(i,j), A(i,:), A(:,j)$ denote the $(i,j)$-th entry, $i$-th row and $j$-th column of $A $ respectively. It is easy to verify that such conjugation satisfies the following properties.
\begin{proposition}
    For $a,b \in \Fd$, $A,B \in M_n(\Fd)$, $C\in GL_n(\Fd)$ the following holds.
    \begin{align*}
        \overline{ab} = \overline{a}\cdot\overline{b}\\
        \overline{a+b} = \overline{a}+\overline{b}\\
        \overline{AB} = \overline{A}\cdot\overline{B}\\
        \overline{C^{-1}} = \overline{C}^{-1}.
    \end{align*}
\end{proposition}
\subsection{The Matrix Multiplication Tensor and Encoding Matrices}
Matrix multiplication is essentially a bilinear map. There are two ways to encode a bilinear map. For a ring $R$, let $f: R^n\times R^m\to R^k$ be a bilinear scheme that performs $r$ ring multiplications. Then there exist $U\in M_{r,n}(R), V\in M_{r,m}(R), W\in M_{r,k}(R)$, such that for every $x\in R^n,y\in R^m$,
\begin{align}
    f(x,y) = W^T((Ux)\odot(Vy)),
\end{align}
where $\odot$ denotes the Hadamard product~\cite{strassenUVW}. We name $(U, V, W)$ the encoding matrices of the map $f$.

Another way is to encode the map into a tensor. From the above encoding matrices $(U, V, W)$, the following tensor encodes the map $f$ as well.
$\mathcal{M} = \sum_{j = 1}^r O_j\otimes P_j\otimes Q_j$
where $O_j = U(j,:), P_j = V(j,:),Q_j = W(j,:)$.

Specifically, in the context of matrix multiplication, let $f: M_{m,n}(R)\times M_{n,p}(R) \to M_{m,p}(R)$ be a matrix multiplication algorithm that performs $r$ multiplication of ring elements. We call this an $\langle m,n,p,r\rangle$-algorithm. In this case, $U \in M_{t,m\cdot n}(R),V\in M_{t,n\cdot p}(R),W\in M_{t,p\cdot m}(R)$. In the tensor decomposition, the vectorization of $O_j \in M_{m,n}(R),P_j\in M_{n,p}(R), Q_j^T \in M_{m,p}(R)$ are rows of $U,V,W$ respectively. Also, they satisfy the Brent equations below.
\begin{proposition}
Let $R$ be a ring. $U, V, W$ are the encoding matrices for an $\langle m,n,p,r\rangle$-algorithm, and $\sum_{j=1}^r O_j\otimes P_j\otimes Q_j$ is the corresponding encoding tensor decomposition if and only if for every $i_1,i_2 \in [m], j_1,j_2 \in [n], k_1,k_2\in[p]$,

\[\sum_{t=1}^r O_t(i_1,j_1)P_t(j_2,k_1)Q_t(k_2,i_2)
    = \delta_{i_1,i_2}\delta_{j_1,j_2}\delta_{k_1,k_2}\]
and equivalently
\[\sum_{t=1}^r U(t,(i_1,j_1))V(t,(j_2,k_1))W(t,(i_2,k_2)) = \delta_{i_1,i_2}\delta_{j_1,j_2}\delta_{k_1,k_2}\]

where $\delta_{i,j} = 1$ if $i=j$ and $0$ otherwise.
\end{proposition}

An example of a matrix multiplication algorithm is Strassen's algorithm for two $2\times 2$ matrix multiplication. For matrices
\begin{align*}
    A = \begin{bmatrix}
a_{11} & a_{12} \\
a_{21} & a_{22} 
\end{bmatrix},\qquad
B = \begin{bmatrix}
b_{11} & b_{12} \\
b_{21} & b_{22} 
\end{bmatrix},\qquad
C = \begin{bmatrix}
c_{11} & c_{12} \\
c_{21} & c_{22} 
\end{bmatrix}
\end{align*}
Strassen's algorithm to compute $AB=C$ is:
\[
\begin{array}{rl|rl}
p_1 &= (a_{1,1} + a_{2,2})(b_{1,1} + b_{2,2}) & 
c_{1,1} &= p_1 + p_4 - p_5 + p_7 \\[6pt]
p_2 &= (a_{2,1} + a_{2,2})(b_{1,1}) &
c_{1,2} &= p_3 + p_5 \\[6pt]
p_3 &= (a_{1,1})(b_{1,2} - b_{2,2}) &
c_{2,1} &= p_2 + p_4 \\[6pt]
p_4 &= (a_{2,2})(b_{2,1} - b_{1,1}) &
c_{2,2} &= p_1 - p_2 + p_3 + p_6 \\[6pt]
p_5 &= (a_{1,1} + a_{1,2})(b_{2,2}) & & \\[6pt]
p_6 &= (a_{2,1} - a_{1,1})(b_{1,1} + b_{1,2}) & & \\[6pt]
p_7 &= (a_{1,2} - a_{2,2})(b_{2,1} + b_{2,2}) & &
\end{array}
\]
Its encoding matrices are 
\begin{align}
    U = \begin{bmatrix}
1 & 0 & 0 & 1 \\
0 & 0 & 1 & 1 \\
1 & 0 & 0 & 0 \\
0 & 0 & 0 & 1 \\
1 & 1 & 0 & 0 \\
-1 & 0 & 1 & 0 \\
0 & 1 & 0 & -1 \\
\end{bmatrix},
 V = \begin{bmatrix}
1 & 0 & 0 & 1 \\
1 & 0 & 0 & 0 \\
0 & 1 & 0 & -1 \\
-1 & 0 & 1 & 0 \\
0 & 0 & 0 & 1 \\
1 & 1 & 0 & 0 \\
0 & 0 & 1 & 1 \\
\end{bmatrix},
W = \begin{bmatrix}
1 & 0 & 0 & 1 \\
0 & 0 & 1 & -1 \\
0 & 1 & 0 & 1 \\
1 & 0 & 1 & 0 \\
-1 & 1 & 0 & 0 \\
0 & 0 & 0 & 1 \\
1 & 0 & 0 & 0 \\
\end{bmatrix}
\end{align}
and the corresponding matrix multiplication tensor
$\sum_{j=1}^7 O_j\otimes P_j\otimes Q_j$ is
\begin{align*}
    &\begin{bmatrix}
1 & 0 \\
0 & 1 
\end{bmatrix}\otimes
\begin{bmatrix}
1 & 0 \\
0 & 1 
\end{bmatrix}\otimes
\begin{bmatrix}
1 & 0 \\
0 & 1 
\end{bmatrix}\\
+&\begin{bmatrix}
0 & 0 \\
1 & 1 
\end{bmatrix}\otimes
\begin{bmatrix}
1 & 0 \\
0 & 0 
\end{bmatrix}\otimes
\begin{bmatrix}
0 & 1 \\
0 & -1 
\end{bmatrix}\\
+&\begin{bmatrix}
1 & 0 \\
0 & 0 
\end{bmatrix}\otimes
\begin{bmatrix}
0 & 1 \\
0 & -1 
\end{bmatrix}\otimes
\begin{bmatrix}
0 & 0 \\
1 & 1 
\end{bmatrix}\\
+&\begin{bmatrix}
0 & 0 \\
0 & 1 
\end{bmatrix}\otimes
\begin{bmatrix}
-1 & 0 \\
1 & 0 
\end{bmatrix}\otimes
\begin{bmatrix}
1 & 1 \\
0 & 0 
\end{bmatrix}\\
+&\begin{bmatrix}
1 & 1 \\
0 & 0 
\end{bmatrix}\otimes
\begin{bmatrix}
0 & 0 \\
0 & 1 
\end{bmatrix}\otimes
\begin{bmatrix}
-1 & 0 \\
1 & 0 
\end{bmatrix}\\
+&\begin{bmatrix}
-1 & 0 \\
1 & 0 
\end{bmatrix}\otimes
\begin{bmatrix}
1 & 1 \\
0 & 0 
\end{bmatrix}\otimes
\begin{bmatrix}
0 & 0 \\
0 & 1 
\end{bmatrix}\\
+&\begin{bmatrix}
0 & 1 \\
0 & -1 
\end{bmatrix}\otimes
\begin{bmatrix}
0 & 0 \\
1 & 1 
\end{bmatrix}\otimes
\begin{bmatrix}
1 & 0 \\
0 & 0 
\end{bmatrix}.
\end{align*}
For matrices of size $2^k$ by $2^k$ for some integer $k$, Strassen's algorithm can be applied to $2^{k-1}$ by $2^{k-1}$ blocks and used to compute matrix multiplication recursively.
Another example is that any standard algorithm for $n\times m$ and $m\times p$ matrix multiplication is a matrix multiplication tensor
\begin{align*}
\sum_{j_1,j_2,j_3=1}^{n,m,p}E_{j_1j_2}^{(n,m)}\otimes{E_{j_2j_3}}^{(m,p)}\otimes E_{j_3j_1}^{(p,n)},
\end{align*}
where these matrices are standard basis matrices of matrix spaces.

\subsection{De-Groote Action}
Given a matrix multiplication tensor $\sum_{j=1}^rO_j\otimes P_j\otimes Q_j$, the following tensors $\sum_{j=1}^rP_j\otimes Q_j\otimes O_j$, $\sum_{j=1}^rP_j^T\otimes O_j^T\otimes Q_j^T$, and $\sum_{j=1}^rXO_jY^{-1}\otimes YP_jZ^{-1}\otimes ZQ_jX^{-1}$ are also matrix multiplication tensors for any invertible $X,Y,Z$~\cite{Groote1978,DG2}.
For a tensor, the following is true
\begin{align}
    \sum_{j=1}^r \alpha_j O_j\otimes P_j\otimes Q_j=\sum_{j=1}^r O_j\otimes \alpha_j P_j\otimes Q_j=\sum_{j=1}^r O_j\otimes P_j\otimes \alpha_j Q_j
\end{align}
We call these the De-Groote actions or De-Groote transformations.

If there exist $\alpha_j, \beta_j,\gamma_j$ such that all entries of $\sum_{j=1}^r \alpha_j O_j\otimes \beta_j P_j\otimes \gamma_j Q_j$ are in $\F$, then we say $\sum_{j=1}^r O_j\otimes P_j\otimes Q_j$ is in $\F$ up to a scalar.

\section{Results}
We show how De-Groote actions can be used to transform a matrix multiplication tensor from $\Fd$ to $\F$, e.g., $\Q[i]\to \Q$. In \ref{Section: alg}. We introduce our algorithm and its theoretical guarantees. In \ref{Section: int}, we show that most tensors in $\Q$ derived here cannot be further transformed to $\Z$.

\subsection{Rational De-Groote Equivalent Point}\label{Section: alg}
Let $\sum_{j=1}^r O_j\otimes P_j \otimes Q_j$ be a matrix multiplication tensor, we know that for any invertible $X, Y, Z$ with compatible sizes, $\sum_{j=1}^r XO_jY^{-1}\otimes YP_j Z^{-1}\otimes Z Q_jX^{-1}$ is still a matrix multiplication tensor because this is a kind of De-Groote action. Dumas, Pernet, and Sedoglavic discovered De-Groote actions that transform the $\langle 4,4,4,48\rangle$ and $\langle3,4,7,63\rangle$ complex tensors in~\cite{NVEE25} to rational tensors~\cite{m44448,m34763}. Their construction involves cleverly leveraging the symmetry of the tensors involved and making educated guesses. By contrast, we give an algorithm that not only tests when it is possible to use De-Groote actions to transform a complex tensor to a real one, but also, if it exists, gives a way to construct it.

Suppose $O_j \in M_{m,n}(\Fd), P_j\in M_{n,p}(\Fd),Q_j \in M_{p,m}(\Fd)$. If there exist $X\in GL_m(\Fd)$, $Y\in GL_n(\Fd)$, $Z\in GL_p(\Fd)$, such that $\sum_{j=1}^r XO_jY^{-1}\otimes YP_j Z^{-1}\otimes Z Q_jX^{-1}$ is in $\F$, then 
\begin{align}
    &(XO_jY^{-1})( YP_j Z^{-1})( Z Q_jX^{-1})=X(O_jP_jQ_j)X^{-1}\\
   &(YP_j Z^{-1})( Z Q_jX^{-1}) (XO_jY^{-1})=Y(P_jQ_jO_j)Y^{-1}\\
    &( Z Q_jX^{-1})(XO_jY^{-1})( YP_j Z^{-1})=Z(Q_jO_jP_j)Z^{-1}
\end{align}
are matrices with all entries in $\F$. This means that there exist invertible $X, Y, Z$ such that 
\begin{align}
    X(O_jP_jQ_j)X^{-1}\in M_{n}(\F),\\
    \,Y(P_jQ_jO_j)Y^{-1}\in M_{m}(\F),\\
    Z(Q_jO_jP_j)Z^{-1}\in M_{p}(\F)
\end{align} for all $j = 1,\dots,r$. 
Theorem \ref{thm: necessary} below shows that if there exists no $S$ such that $SO_jP_jQ_j = \overline{O_jP_jQ_j}S$ and $S\overline{S} = I$ for all $j$, then there is no such $X$. The existence of $Y, Z$ is similar. Finding such an $S$ is equivalent to solving a system of linear equations. Once this $S$ is found, then inspired by the proof of Theorem \ref{thm: sufficiency} below, we can construct $X$ from $S$ by solving a system of linear equations. 

\begin{lemma}\label{lemma: necessary}
Let $A\in M_{n}(\Fd)$. If there exists $X\in GL_n(\Fd)$, such that $XAX^{-1} \in M_{n}(\mathbb{\F})$, then there exist $S\in GL_n(\Fd)$, such that $SA = \bar{A}S$ and $S\overline{S} = I$.
\end{lemma}
\begin{proof}
    Let $R = XAX^{-1}$. Since $R$ is in $\F$,
    \begin{align}
        XAX^{-1} = R=\bar{R} = \overline{XAX^{-1}} =\overline{X}\cdot\overline{A}\cdot \overline{X}^{-1}
    \end{align}
    The proof is complete taking $S = \overline{X}^{-1}X$.
\end{proof}

Applying this lemma to our context is the following.
\begin{theorem}\label{thm: necessary}
    If there is no $S$ such that $S(O_jP_jQ_j) = \overline{O_jP_jQ_j}S$ with $\overline{S}S =I$ simultaneously for $j =1,\dots,r$, or similarly with $P_jQ_jO_j$ or $Q_jO_jP_j$, then there are no $X,Y,Z$ such that $XO_jY^{-1}\otimes YP_j Z^{-1}\otimes Z Q_jX^{-1}$ are all in $\F$.
\end{theorem}

\begin{theorem}\label{thm: sufficiency}
   For $A\in M_{n,n}(\Fd)$, if there exists $S\in GL_n(\Fd)$ such that $SA = \overline{A}S $ and $S\overline{S} =I$, then there exists $X\in GL_n(\Fd)$, such that $XAX^{-1}\in M_{n}(\mathbb{\F})$.
\end{theorem}
We use two lemmas to prove Theorem 2.

\begin{lemma}\label{lemma: anti-linear}
    In this lemma, we refer to vectors as row vectors. Let $S\in GL_n(\Fd)$ be a matrix such that $S\overline{S} = I$. Then $f:\Fd^n\to \Fd^n$, $f(x) = \overline{x}S$ is an anti-linear involution, i.e. for $\alpha,\beta\in\Fd$ and $x,y\in \Fd^n$
        \begin{align}
            f(\alpha x+\beta y) =\overline{\alpha}f(x) + \overline{\beta}f(y) \qquad &\text{(anti-linearity)}\\
            f\circ f = id_{\Fd^n}\qquad & \text{(involution)}.
        \end{align}
\end{lemma}
\begin{proof}
Let $\alpha,\beta\in\Fd$ and $x,y\in \Fd^n$.

Anti-linearity:
\[
f(\alpha x+\beta y) = (\overline{\alpha x+\beta y})S = \overline{\alpha x}S+\overline{\beta y}S=\overline{\alpha}f(x) + \overline{\beta}f(y)
\]

Involution:
\[
f(f(x)) = f(\overline{x}S) =\overline{\overline{x}S }S  =x \overline{S}S = x
\]

\end{proof}

\begin{lemma}\label{lemma: subspace}
    Let $f:\Fd^n \to \Fd^n$ be an anti-linear involution and let $\Fix(f) =\{x\in \Fd: f(x) = x\}$ be the set of fixed points of $f$. Then
        $\Fix(f)$ is an $n$-dimensional subspace \textbf{over} $\mathbb{\F}$.
\end{lemma}
\begin{proof}
    For $a,b \in \Fix(f)$, and $\alpha \in \F$,
    \begin{align*}
        &f(a+b) = f(a) + f(b) = a + b\implies a+b \in \Fix(f)\\
        &f(\alpha b) = \alpha f(b) = \alpha b \implies \alpha b\in \Fix(f).
    \end{align*}
    This shows that $\Fix(f)$ is a subspace over $\F$. Next, we show that $\Fix(f) \bigoplus \sqrt{d}\Fix(f) = \Fd^n$, where $\sqrt{d}\Fix(f) = \{\sqrt{d}a:a\in \Fix{(f)}\}$.\\
    For any $a \in \Fd^n$, $a = a_1 + a_2$, where
    \begin{align*}
        a_1 = \frac{a+f(a)}{2},\qquad a_2 = \frac{a-f(a)}{2}.
    \end{align*}
    the following holds
    $f(a_1) = f\left(\frac{a+f(a)}{2}\right) = \frac{1}{2}[f(a) + f(f(a))] = a_1$, and therefore, $a_1\in \Fix(f)$. In addition, 
    \[
    f\left(\frac{\sqrt{d}}{d}a_2\right) = -\frac{\sqrt{d}}{d}f\left(\frac{a - f(a)}{2}\right)= -\frac{\sqrt{d}}{d}\frac{1}{2}[f(a) -a ] = \frac{\sqrt{d}}{d}a_2
    \]
    thus, $a_2 \in\sqrt{d}\Fix(f)$.
    If $a \in \Fix(f) \cap\sqrt{d}\Fix(f)$
    \begin{align*}
         \frac{\sqrt{d}}{d}a = f\left(\frac{\sqrt{d}}{d}a\right) = -\frac{\sqrt{d}}{d}f(a) = -\frac{\sqrt{d}}{d}a.
    \end{align*}
    Thus, the intersection is the trivial space, and this is a direct sum, which implies that the dimension of $\Fix(f)$ is $n$. 
\end{proof}

\begin{proof}[Proof of Theorem \ref{thm: sufficiency}]
        This theorem appears in~\cite{Horn_Johnson_2012} and some other textbooks. The proof uses the Jordan normal form. However, since the proof idea is a key component to construct $X$ from $S$, we state a constructive proof here.
        
        From the proof of Lemma \ref{lemma: necessary}, we see $S$ and $X$ satisfy \begin{align}
            \overline{X}S = X
        \end{align}
        meaning rows $x_j$ of $X$ satisfy $\overline{x_j}S = x_j$. This serves as the intuition for the construction.
        Define $f: \Fd^n \to \Fd^n, f(x) = \overline{x}S$. By Lemma \ref{lemma: anti-linear}, this is an anti-linear involution, and by Lemma \ref{lemma: subspace}, 
        $\Fix(f)$ is an $n$-dimensional subspace \textbf{over} $\mathbb{\F}$.
        Let $x_1,x_2,\dots,x_n$ be an $\F$-basis of $\Fix(f)$. Form $X = [x_1^T,x_2^T,\cdots,x_n^T]^T$, and this $X$ satisfies the desired property as
    \begin{align*}
            \overline{XAX^{-1}} &= \overline{X}\cdot\overline{A}\cdot\overline{X}^{-1}\\
            &=\overline{X}\cdot (S A S^{-1})\cdot\overline{X}^{-1}\\
            & = (\overline{X}S)A(\overline{X}S)^{-1}\\
            & = XAX^{-1}~,
        \end{align*}
        so by necessity $XAX^{-1}$ is in $M_n(\F)$. 
\end{proof}
The above theorems are summarized into Algorithm 1.

\begin{algorithm}[ht!]
        \caption{Find a $\F$ point in the orbit from an $\Fd$ tensor}
        \textbf{Input:} A matrix multiplication tensor $\sum_{j=1}^r O_j\otimes P_j \otimes Q_j$ with entries in $\Fd$
        \begin{algorithmic}
         \item 1. For each $j=1,...,r$, form $M_j = O_jP_jQ_j$.
         \item 2. Solve $SM_j = \overline{M_j}S$, $\overline{S}S=I$ $j=1,\dots,r$ simultaneously.
         \item 3. If $S$ exists\\
         \qquad Go to 4.\\
         \quad Else\\
         \qquad Output Null
        \item 4. Compute the basis $\{x_1,\dots,x_n\}$ for the  fixed point space of $x\to \overline{x}S$.
        \item 5. Form $X = [x_1;\dots;x_n]$.
        \item 6. Repeat 1-5 with $M_j =  P_jQ_jO_j$ and $M_j =  Q_jO_jP_j$ to form $Y,Z$.
        \end{algorithmic}
\end{algorithm}

At first glance, this algorithm only ensures that the matrices $UO_jP_jQ_jU^{-1}$, $VP_jQ_jO_jV^{-1}$, $WQ_jO_jP_jW^{-1}$ are all in $\F$. In principle, it might happen that for some $j$, $UO_jV^{-1}$ or $VP_jW^{-1}$ or $WQ_jU^{-1}$ are not in the form of $\alpha R$, where $\alpha\in\Fd$ and $R\in M(\F)$. For example,
\begin{align*}\label{eq: problem}
    XO_1Y^{-1} = \begin{bmatrix}
i & 0 \\
0 & 1 
\end{bmatrix}, YP_1Z^{-1} = \begin{bmatrix}
-i & 0 \\
0 & 1 
\end{bmatrix}, ZQ_1X^{-1} = \begin{bmatrix}
1 & 0 \\
0 & 1 
\end{bmatrix}
\end{align*}
satisfy that $XO_1P_1Q_1X^{-1}$ $YP_1Q_1O_1Y^{-1},ZQ_1O_1P_1Z^{-1}$ are all identity matrices, but it is not in $\F$ up to a scalar.

However, this cannot happen if the $S$ in the algorithm is unique up to a scalar $\alpha$ with $\alpha\overline{\alpha} = 1$, i.e. if $S, S'$ are both such that $SO_jP_jQ_j = \overline{O_jP_jQ_j}S$ and $S\overline{S} = I$, then $S = \alpha S'$ for some $\alpha\in \Fd$ with $\alpha\overline{\alpha} = 1$. To show this, we first prove a lemma.

\begin{lemma}\label{prop: Base Independence}
    For a matrix multiplication tensor $\sum_{j=1}^rO_j\otimes P_j\otimes Q_j$ over $\Fd$, suppose that there exists one triple $X, Y, Z$ that is derived from the algorithm in STEP 4, such that $XO_jY^{-1} = \alpha_1R_1$, $YP_jZ^{-1} = \alpha_2R_2$, $ZQ_jX^{-1} = \alpha_3 R_3$, where $\alpha_1,\alpha_2,\alpha_3\in\Fd$ with $\alpha_1\alpha_2\alpha_3\in\F$, and $R_1, R_2, R_3 \in M_n(\mathbb{F})$. Then, for every possible $\widetilde{X},\widetilde{Y},\widetilde{Z}$ derived from the algorithm, $\widetilde{X}O_j\widetilde{Y}^{-1}$, $\widetilde{Y}P_j\widetilde{Z}^{-1}$ and $\widetilde{Z}Q_j\widetilde{X}^{-1}$ have the same form.
\end{lemma}

\begin{proof}
    Both $X$ and $\widetilde{X}$ satisfy
    \begin{align*}
        \overline{X}S = X\qquad  \overline{\widetilde{X}}S = \widetilde{X}.
    \end{align*}
    It follows that
    \begin{align*}
        \overline{\widetilde{X}X^{-1}} = \overline{\widetilde{X}}\overline{X^{-1}}= \overline{\widetilde{X}}(\overline{\overline{X}S)^{-1}}=\overline{\widetilde{X}}\overline{S^{-1}\overline{X}^{-1}}= (\overline{\widetilde{X}}S)X^{-1}=\widetilde{X}X^{-1}.
    \end{align*}
    Thus, $\widetilde{X}X^{-1}$ is in $\F$. Similar arguments show that $\widetilde{Y}Y^{-1}$ and $\widetilde{Z}Z^{-1}$ are real. Now
    \begin{align*}
        \widetilde{X}O_j\widetilde{Y}^{-1} = \widetilde{X}X^{-1}XO_jY^{-1}(\widetilde{Y}Y^{-1})^{-1} = \alpha_1 \widetilde{X}X^{-1}R_1(\widetilde{Y}Y^{-1})^{-1}.
    \end{align*}
    Repeating this with $\widetilde{Y}P_j\widetilde{Z}^{-1}$ and $\widetilde{Z}Q_j\widetilde{X}^{-1}$ completes the proof.
\end{proof}
This lemma shows that the choice of basis in STEP 4 of the algorithm does not matter. If one choice of basis works, then all bases work. The following theorem shows that the previous potential problem does not occur.
\begin{theorem}
        For a matrix multiplication tensor $\sum_{j=1}^rO_j\otimes P_j\otimes Q_j$ over $\Fd$, suppose there exist $X,Y,Z$ such that $XO_jY^{-1}\otimes VP_jW^{-1} \otimes WQ_jU^{-1}$ are all in $\F$ up to a scalar. If $S$ in the algorithm exists and is unique up to a scalar $\alpha$ with $\alpha\overline{\alpha} = 1$, then $\widetilde{X}O_j\widetilde{Y}^{-1}\otimes \widetilde{Y}P_j\widetilde{Z}^{-1} \otimes \widetilde{Z}Q_j\widetilde{X}^{-1}$ are all in $\F$ for any $\widetilde{X},\widetilde{Y},\widetilde{Z}$ computed by the algorithm.
    \end{theorem}

\begin{proof}
    Suppose $X,Y,Z$ are such that $\sum_{j=1}^r XO_jY^{-1}\otimes YP_jZ^{-1}\otimes ZQ_jX^{-1}$ is in $\F$. Then $S = \overline{X^{-1}}X$ satisfies $SO_jP_jQ_j = \overline{O_jP_jQ_j}S$ and $S\overline{S} = I$. If we assume $S$ is unique in the above sense, then $\widetilde{S}$ and $\widetilde{X}$ computed by the algorithm satisfy $\widetilde{S}= \alpha S$ for some $\alpha \in\Fd$, and $\overline{\widetilde{X}}\widetilde{S} = \widetilde{X}$.
By Theorem \ref{thm: sufficiency}, there exists $T\in GL_n(\Fd)$, such that $\overline{T}^{-1}T = \alpha I$. It follows that
\begin{align}
    \widetilde{S} =\alpha S = (\overline{TX})^{-1}TX 
\end{align}
This means the rows of $X$ and $\widetilde{X}$ form bases for the fixed point space of $x\to \overline{x}S$.
Together with Proposition \ref{prop: Base Independence}, this gives the result.
    \end{proof}

\subsection{Integer algorithms}\label{Section: int}
Once a $\Q[\sqrt{d}]$ tensor is transformed to a $\Q$ tensor, one may ask further if there exist $X,Y,Z$ such $\sum_{j=1}^r XO_jY^{-1}\otimes YP_j Z^{-1}\otimes Z Q_jX^{-1}$ is an integer tensor for even faster computation. In this section, we show a way to check existence by examining the trace of $O_jP_jQ_j$.

If in a matrix multiplication tensor there exists $j$ such that trace of $O_jP_jQ_j$ is not an integer, then by Proposition \ref{Trace} below, there do not exist $X, Y, Z$ such that $\sum_{j=1}^r XO_jY^{-1}\otimes YP_j Z^{-1}\otimes Z Q_jX^{-1}$ is a tensor with only integer matrices (up to scalar). This actually shows that the $\langle 3,4,7,63\rangle$ complex algorithm found by the Google DeepMind team and the $\langle3,3,6,40\rangle$ $\Z[1/2]$ algorithm found by Smirnov cannot be transformed into integer algorithms, because in both cases $\operatorname{Trace}(O_jP_jQ_j) \not \in \Z$ for some $j$.

If all $O_jP_jQ_j$ have integer trace, for example, all $O_jP_jQ_j$'s in $\langle4,4,4,48\rangle$ tensor over $\Q[i]$ found by Google DeepMind team have trace 1 or 2, then we can next examine the $\operatorname{Trace}(O_{j_1}P_{j_1}Q_{j_1}\cdot O_{j_2}P_{j_2}Q_{j_2})$ for all $j_1,j_2 = 1,\dots,r$. As it turns out, there exist $j_1,j_2$ such that 
\begin{align*}
    &\operatorname{Trace}(O_{j_1}P_{j_1}Q_{j_1})=1\\
    &\operatorname{Trace}(O_{j_2}P_{j_2}Q_{j_2})=1\\
    &\operatorname{Trace}(O_{j_1}P_{j_1}Q_{j_1}\cdot O_{j_2}P_{j_2}Q_{j_2}) = 1/2
\end{align*}
This implies that the $\Z$ equivalent form of this algorithm does not exist. Similarly, the $\langle 2, 4, 5, 32\rangle$ algorithm over $\Z[1/2]$ by Hopcroft and Kerr does not have a $\Z$ equivalent form for the same reason.

Finally, the $\langle2,4,4,26\rangle$ algorithm over $\Z[1/2]$ also does not have a $\Z$ equivalent. This algorithm satisfies that for all $j=1,\dots,r$ and for all pairs $j_1,j_2 = 1,\dots,r$
\begin{align*}
    &\operatorname{Trace}(O_{j_1}P_{j_1}Q_{j_1})\in \Z\\
    &\operatorname{Trace}(O_{j_1}P_{j_1}Q_{j_1}\cdot O_{j_2}P_{j_2}Q_{j_2})\in \Z.
\end{align*}
However, there exist $j_1,j_2,j_3$ such that
\begin{align*}
   \operatorname{Trace}(O_{j_1}P_{j_1}Q_{j_1}\cdot O_{j_2}P_{j_2}Q_{j_2}\cdot O_{j_3}P_{j_3}Q_{j_3})\not\in \Z,
\end{align*}
which shows the non-existence of a $\Z$ equivalent form.

In general, we can use the following Proposition.
\begin{proposition}\label{Trace}
    Let $A_1,\dots,A_m\in M_n(\C)$. Suppose that there exist $j_1,\dots,j_k$ such that $\operatorname{Trace}(\prod_{\ell=1}^k A_{j_\ell})\not\in \mathbb{Z}$, then there does not exist $X$ invertible, such that $XA_jX^{-1} \in  M_n(\mathbb{Z})$ for every $j$.
\end{proposition}
\begin{proof}
    Suppose that there is an $X$ such that $XA_jX^{-1} \in M_n(\Z)$ for every $j$. Then for every indices $j_1,\dots,j_k$, $$\prod_{\ell=1}^k (XA_{j_\ell}X^{-1}) = X(\prod_{\ell=1}^k A_{j_\ell})X^{-1}\in M_n(\Z).$$
    Hence $\operatorname{Trace(\prod_{\ell=1}^k A_{j_\ell})} \in \Z$. This proves the contrapositive of the proposition.
\end{proof}

\section{Applications}
We applied our algorithm to the $\langle 4,4,4,48\rangle$ and $\langle 3,4,7,63\rangle$ algorithms over $\Q[i]$ in~\cite{NVEE25}, $\langle 4,9,4,104\rangle$ algorithm over $\Q[\sqrt{161}]$ in ~\cite{m494104}, and another $\langle 4,4,4,48\rangle$ algorithm over $\C$ in~\cite{m44448c}. 
We also checked if the $\langle 3, 3, 6, 40\rangle$ $\langle2,4,5,32\rangle$ and $\langle2,4,4,26\rangle$ algorithms over $\Z[1/2]$ in~\cite{m33640,Hopcroft1971,Ballard2015} can be transformed into $\Z$ tensors. Among these, two $\Q[i]$ tensors can be transformed to a tensor in $\Q$ as Dumas, Pernet, and Sedoglavic showed in \cite{m44448,m34763}, and we can show that the $\Q[\sqrt{161}]$ tensor cannot be transformed to a rational one using De-Groote action. The $\langle 4,4,4,48\rangle$ $\C$ tensor in~\cite{m44448c} cannot be transformed to an $\R$ tensor using De-Groote actions as well. 
The information above is summarized in the following table.

\begin{table}[h!]
    \centering
    \begin{tabular}{|c|c|c|c|}
        \hline
        Algorithm & Original Ring & New Ring & Non-existence\\
        \hline
        $\langle4,4,4,48\rangle$~\cite{NVEE25} & $\Q[i]$  & $\Z[1/8]$ &$\Z$\\
        \hline
        $\langle3,4,7,63\rangle$~\cite{NVEE25} & $\Q[i]$  & $\Q$ &$\Z$\\
        \hline
        $\langle4,4,4,48\rangle$~\cite{m44448c} & $\C$  & $\C$ &$\R$\\
        \hline
        $\langle4,9,4,104\rangle$~\cite{m494104} & $\Q[\sqrt{161}]$  & $\Q[\sqrt{161}]$ &$\Q$\\
        \hline
        $\langle3,3,6,40\rangle$~\cite{m33640} & $\Z[1/2]$  & $\Z[1/2]$ &$\Z$\\
        \hline
         $\langle2,4,5,32\rangle$~\cite{Hopcroft1971} & $\Z[1/2]$  & $\Z[1/2]$ &$\Z$\\
        \hline
         $\langle2,4,4,26\rangle$~\cite{BB15} & $\Z[1/2]$  & $\Z[1/2]$ &$\Z$\\
        \hline
    \end{tabular}
    \label{table:1}
\end{table}

The detailed steps of how our algorithm is applied to the $\langle 4,4,4,48\rangle$ tensor in \cite{NVEE25} are the following.

In step 1 and 2, after forming $O_jP_jQ_j$, the null space is one dimensional: 
\begin{align*}
\operatorname{Span}\left\{\begin{bmatrix}
0 & 0 & 0 & 1\\
0 & 0 & i & 0\\
0 & i & 0 & 0\\
1 & 0 & 0 & 0
\end{bmatrix}\right\}
\end{align*}
Thus, $S$ is in form
\begin{align*}
\begin{bmatrix}
0 & 0 & 0 & \alpha\\
0 & 0 & i\alpha & 0\\
0 & i\alpha & 0 & 0\\
\alpha & 0 & 0 & 0
\end{bmatrix}, \alpha \in \Q[i], |\alpha| = 1.
\end{align*}
The $X$~\cite{m44448} satisfies $\overline{X}S=X$ in STEP 4. Another feasible $X$ is
\begin{align*}
X = \begin{bmatrix}
1 & 0 & 0 & 1\\
0 & 1 & i & 0\\
0 & i & 1 & 0\\
i & 0 & 0 & -i
\end{bmatrix}
\end{align*}

For $P_jQ_jO_j$, the null space is also one-dimensional
\begin{align*}
\operatorname{Span}\left\{\begin{bmatrix}
-1 & 0 & 0 & 0\\
0 & -1 & 0 & 0\\
0 & 0 & -1 & 0\\
0 & 0 & 0 & 1
\end{bmatrix}\right\}
\end{align*}
Thus, $S$ is in form
\begin{align*}
\begin{bmatrix}
-\alpha & 0 & 0 & 0\\
0 & -\alpha & 0 & 0\\
0 & 0 & -\alpha & 0\\
0 & 0 & 0 & \alpha
\end{bmatrix} \alpha \in \Q[i], |\alpha| = 1.
\end{align*}
One feasible $Y$ is 
\begin{align*}
Y = \begin{bmatrix}
i & 0 & 0 & 0\\
0 & i & 0 & 0\\
0 & 0 & i & 0\\
0 & 0 & 0 & 1
\end{bmatrix}
\end{align*}
For  $Q_jO_jP_j$, $S =I$, so $Z$ can simply be $I$. The rational algorithm derived from transforming the $\langle4,4,4,48\rangle$ complex algorithm in~\cite{NVEE25} by taking 
\begin{align*}
     &X = \begin{bmatrix}
1 & 0 & 0 & 1\\
0 & 1 & i & 0\\
0 & i & 1 & 0\\
i & 0 & 0 & -i
\end{bmatrix} , \\
&Y = \begin{bmatrix}
-i & i & i & 0\\
0 & -i & 0 & 0\\
0 & 0 & -i & 0\\
0 & 0 & 0 & 1
\end{bmatrix}, \\
&Z = \begin{bmatrix}
1/2 & -1/2 & 1/2 & 0\\
1/2 & 1/2 & -1/2 & 0\\
-1/2 & 1/2 & 1/2 & 0\\
0 & 0 & 0 & 1
\end{bmatrix}
\end{align*}
is in the appendix.
Here, $X$, $Y$, $Z$ are arbitrarily chosen to illustrate that the choice of basis does not matter. Based on Lemma \ref{prop: Base Independence}, if this particular triple works, then any $X, Y, Z$ generated by the algorithm will work. 

\section{Open Problems}
\begin{enumerate}
  \item In this paper, we find a sufficient condition for a $\Z$ equivalent algorithm not to exist. It is useful to find a necessary and sufficient condition for an integer algorithm to be achieved using De-Groote actions.
  \item We argue that a problem like \ref{eq: problem} will not occur if $S$ is unique up to a scalar. Although $S$ is unique in all cases we see, it is interesting to argue whether the above problem can be avoided.
  \item Many fast matrix algorithms do not have an equivalent form over a reduced ring. 
  Another future direction is using the alternative basis method to bridge this gap.
  The alternative basis method was introduced by Karstadt and Schwartz to reduce the leading coefficient of complexity of fast matrix multiplication algorithms by sparsifying the encoding matrices \cite{AB1}. 
\end{enumerate}

\section{Acknowledgment}
Shuncheng Yuan acknowledges support from UC MRPI M25PL9060. 
We also thank Ioana Dumitriu for helpful comments for improving the clarity of the paper.
This project was supported by the European Research Council (ERC) under the European Union’s Horizon 2020 research and innovation programme (grant agreement No. 818252, 101113120, 101138056, 101142020), the Israel Science Foundation (grant No. 1354/23, 514/20), the Science Accelerator, the Frontiers in Science initiative of the Ministry of Innovation, Science and Technology, and from the Len Blavatnik and the Blavatnik Family Foundation.
This work was done in part while the authors were visiting the Simons Institute for the Theory of Computing.

\printbibliography
\appendix
\section{A Rational $\langle4,4,4,48\rangle$ Tensor}
\begin{longtable}{|c|c|c|c|}
\hline
 & $O_j$ & $P_j$ & $Q_j$ \\ \hline
\endfirsthead
\hline
 & $O_j$ & $P_j$ & $Q_j$ \\ \hline
\endhead

$j=1 $& $\begin{matrix}
1 & 2 & 1 & 0\\
1 & 2 & 1 & 0\\
1 & 2 & 1 & 0\\
-1 & -2 & -1 & 0
\end{matrix}$ & $\begin{matrix}
-1/2 & -1/2 & 0 & 0\\
-1/2 & -1/2 & 0 & 0\\
1/2 & 1/2 & 0 & 0\\
1/2 & 1/2 & 0 & 0
\end{matrix}$ &$\begin{matrix}
0 & -1/2 & 0 & 1/2\\
0 & 0 & 0 & 0\\
0 & 1/2 & 0 & -1/2\\
0 & -1/2 & 0 & 1/2
\end{matrix}$ \\ \hline

 $j=2 $& $\begin{matrix}
-1 & -1 & -1 & -1\\
-1 & -1 & -1 & -1\\
1 & 1 & 1 & 1\\
1 & 1 & 1 & 1
\end{matrix}$ & $\begin{matrix}
0 & -1/2 & -1/2 & -1/2\\
0 & 1/2 & 1/2 & 1/2\\
0 & 1/2 & 1/2 & 1/2\\
0 & 1/2 & 1/2 & 1/2
\end{matrix}$ &$\begin{matrix}
-1/4 & 0 & 1/4 & 0\\
-1/4 & 0 & 1/4 & 0\\
1/4 & 0 & -1/4 & 0\\
-1/2 & 0 & 1/2 & 0
\end{matrix}$ \\ \hline

$j=3 $& $\begin{matrix}
 0&  1&  -1& 0\\
 0&  1&  -1& 0\\
 0&  -1&  1& 0\\
 0&  1&  -1&0 
\end{matrix}$ & $\begin{matrix}
 0& 1/2 & 1/2 & 0\\
 0&  -1/2&  -1/2& 0\\
 0&  1/2&  1/2& 0\\
 0&  1/2&  1/2& 0
\end{matrix}$ &$\begin{matrix}
 0&  0& -1/2 & 1/2\\
 0&  0&  1/2& -1/2\\
 0&  0&  0& 0\\
 0& 0 &  0& 0
\end{matrix}$ \\ \hline

$j=4 $& $\begin{matrix}
 1&  2&  0& -1\\
 -1&  0&  -2& 1\\
 1&  2&  0& -1\\
 -1&  0&  -2& 1
\end{matrix}$ & $\begin{matrix}
 0&  -1/2&  1/2& 0\\
 -1/2&  -1/2&  -1&0 \\
 1/2&  1/2&  1&0 \\
 0&  1/2& -1/2 & 0
\end{matrix}$ &$\begin{matrix}
3/8 & -1/8 & 3/8 & - 1/8 \\
 -3/8&  1/8&  -3/8& 1/8\\
 -1/8&  -1/8& -1/8 &-1/8 \\
 1/4&  1/4& 1/4 & 1/4
\end{matrix}$ \\ \hline

$j=5$ &
$\begin{matrix}
-1 & 0 & -1 & 0\\
1 & 0 & 1 & 0\\
1 & 0 & 1 & 0\\
1 & 0 & 1 & 0
\end{matrix}$ &
$\begin{matrix}
0 & 3/2 & -3/2&-3/2\\
0 & -1/2 & 1/2&1/2\\
0 & -1/2 & 1/2&1/2\\
0 & 1/2 & -1/2&-1/2
\end{matrix}$ &
$\begin{matrix}
 0 & 0 & 0 & 0\\
 0 & 0 & 0 & 0\\
 0 & 0 & 0 & 0\\
 0 & -1/2 & 0 & -1/2
\end{matrix}$ \\ \hline

$j=6$ &
$\begin{matrix}
0 & 0 & 1 & -1\\
0 & 0 & 1 & -1\\
0 & 0 & -1 & 1\\
0 & 0 & -1 & 1
\end{matrix}$ &
$\begin{matrix}
0 & -1/2 & -1/2&-1/2\\
0 &  1/2 &  1/2&1/2\\
0 &  1/2 &  1/2&1/2\\
0 & -1/2 & -1/2&-1/2
\end{matrix}$ &
$\begin{matrix}
 0 & 0 & 3/4 &3/4\\
0 & 0 & -1/4 & -1/4\\
 0 & 0 & -1/4 & -1/4\\
 0 & 0 & 0 & 0
\end{matrix}$ \\ \hline

$j=7$ &
$\begin{matrix}
-1/2 & -1/2 & -1/2 &  1/2\\
-1/2 & -1/2 & -1/2 &  1/2\\
 1/2 &  1/2 &  1/2 & -1/2\\
-1/2 & -1/2 & -1/2 &  1/2
\end{matrix}$ &
$\begin{matrix}
0 & -3 & -3 & 0\\
0 &  1 &  1 &0\\
0 &  1 &  1&0\\
0 &  1 &  1&0
\end{matrix}$ &
$\begin{matrix}
 -1/4 & -1/4 & 0 & 0\\
 1/4 &  1/4 & 0 & 0\\
 1/4 &  1/4 & 0 & 0\\
-1/2 & -1/2 & 0 & 0
\end{matrix}$ \\ \hline

$j=8$ &
$\begin{matrix}
-1 & 0 & -1 & 0\\
1 & 0 & 1 & 0\\
1 & 0 & 1 & 0\\
-1 & 0 & -1 & 0
\end{matrix}$ &
$\begin{matrix}
-1/2 & -1/2 & 0 & 1/2\\
 1/2 &  1/2 & 0& -1/2\\
-1/2 & -1/2 & 0& 1/2\\
-1/2 & -1/2 & 0&1/2
\end{matrix}$ &
$\begin{matrix}
 0 & 0 & 0 & 0\\
 0 & 0 & 0 & 0\\
0 & 0 & 0 & 0\\
 0 & 0 &  1/2 & -1/2
\end{matrix}$ \\ \hline

$j=9$ &
$\begin{matrix}
1/2 & 1 & 0 & -1/2\\
1/2 & 1 & 1 & 1/2\\
-1/2 & -1 & 0 & 1/2\\
1/2 & 1 & 1 & 1/2
\end{matrix}$ &
$\begin{matrix}
0 & -1 & -1&0\\
0 & 1 & -1&-1\\
0 & 1 & 1&0\\
0 & 1 & 1&0
\end{matrix}$ &
$\begin{matrix}
 1/8 & -3/8 & -1/8 & -3/8\\
 1/8 & 1/8 & -1/8 & 1/8\\
 -3/8 & 1/8 & 3/8 & 1/8\\
 1/4 & -1/4 & -1/4 & -1/4
\end{matrix}$ \\ \hline

$j=10$ &
$\begin{matrix}
1 & 1 & 1 & 1\\
-1 & -1 & -1 & -1\\
1 & 1 & 1 & 1\\
1 & 1 & 1 & 1
\end{matrix}$ &
$\begin{matrix}
1/2 & 1/2 & 1&1/2\\
1/2 & 1/2 & 1&1/2\\
-1/2 & -1/2 & -1&-1/2\\
1/2 & 1/2 & 1&1/2
\end{matrix}$ &
$\begin{matrix}
 1/4 & -1/4 & 0 & 0\\
 1/4 & -1/4 & 0 & 0\\
 -1/4 & 1/4 & 0 & 0\\
 1/2 & -1/2 & 0 & 0
\end{matrix}$ \\ \hline

$j=11$ &
$\begin{matrix}
-1 & 0 & -1 & 0\\
-1 & 0 & -1 & 0\\
1 & 0 & 1 & 0\\
1 & 0 & 1 & 0
\end{matrix}$ &
$\begin{matrix}
0 & 3/2 & 3/2&3/2\\
0 & -1/2 & -1/2&-1/2\\
0 & -1/2 & -1/2&-1/2\\
0 & -1/2 & -1/2& -1/2
\end{matrix}$ &
$\begin{matrix}
 0 & 0 & 0 & 0\\
 0 & 0 & 0 & 0\\
0 & 0 & 0 & 0\\
 -1/2 & -1/2 & 0 & 0
\end{matrix}$ \\ \hline

$j=12$ &
$\begin{matrix}
1/2 & 1 & 1 & -1/2\\
-1/2 & -1 & -1 & 1/2\\
1/2 & 1 & 0 & 1/2\\
-1/2 & -1 & 0 & -1/2
\end{matrix}$ &
$\begin{matrix}
-1/2 & -1/2 & -1&0\\
-1/2 & -1/2 & 0&1/2\\
1/2 & 1/2 & 1&0\\
-1/2 & -1/2 & -1&0
\end{matrix}$ &
$\begin{matrix}
 -5/4 & 5/4 & -1/4 & 1/4\\
  1/4 & -1/4 & 1/4 & -1/4\\
 3/4 & -3/4 & -1/4 & 1/4\\
 -1/2 & 1/2 & 1/2 & -1/2
\end{matrix}$ \\ \hline

$j=13$ &
$\begin{matrix}
0 & -1 & 1 & 0\\
0 & 1 & -1 & 0\\
0 & 1 & -1 & 0\\
0 & 1 & -1 & 0
\end{matrix}$ &
$\begin{matrix}
0 & 1/2 & -1/2&-1/2\\
0 & 1/2 & -1/2&-1/2\\
0 & -1/2 & 1/2&1/2\\
0 & 1/2 & -1/2&-1/2
\end{matrix}$ &
$\begin{matrix}
 1/2 & -1/2 & 0 & 0\\
 -1/2 & 1/2 & 0 & 0\\
 0 & 0 & 0 & 0\\
 0 & 0 & 0 & 0
\end{matrix}$ \\ \hline

$j=14$ &
$\begin{matrix}
0 & -1 & 1 & 0\\
0 & -1 & 1 & 0\\
0 & -1 & 1 & 0\\
0 & -1 & 1 & 0
\end{matrix}$ &
$\begin{matrix}
0 & 1/2 & -1/2&0\\
0 & -1/2 & 1/2&0\\
0 & 1/2 & -1/2&0\\
0 & -1/2 & 1/2&0
\end{matrix}$ &
$\begin{matrix}
 -1/2 & 0 & -1/2 & 0\\
 1/2 & 0 & 1/2 & 0\\
 0 & 0 & 0 & 0\\
 0 & 0 & 0 & 0
\end{matrix}$ \\ \hline

$j=15$ &
$\begin{matrix}
-1/2 & 0 & -1 & -1/2\\
-1/2 & 0 & -1 & -1/2\\
-1/2 & -1 & 0 & -1/2\\
1/2 & 1 & 0 & 1/2
\end{matrix}$ &
$\begin{matrix}
0 & 1 & 1 &0\\
1 & 1 & 0&0\\
-1 & -1 & 0&0\\
0 & 1 & 1&0
\end{matrix}$ &
$\begin{matrix}
 3/8 & 3/8 & 1/8 & -1/8\\
 -1/8 & -1/8 & -3/8 & 3/8\\
 -1/8 & -1/8 & 1/8 & -1/8\\
 1/4 & 1/4 & -1/4 & 1/4
\end{matrix}$ \\ \hline

$j=16$ &
$\begin{matrix}
0 & 0 & 1 & -1\\
0 & 0 & -1 & 1\\
0 & 0 & -1 & 1\\
0 & 0 & 1 & -1
\end{matrix}$ &
$\begin{matrix}
1/2 & 1/2 & 0& -1/2\\
1/2 & 1/2 & 0&-1/2\\
-1/2 & -1/2 & 0&1/2\\
1/2 & 1/2 & 0&-1/2
\end{matrix}$ &
$\begin{matrix}
 -3/4 & 3/4 & 0 & 0\\
 1/4 & -1/4 & 0 & 0\\
 1/4 & -1/4 & 0 & 0\\
 0 & 0 & 0 & 0
\end{matrix}$ \\ \hline

$j=17$ &
$\begin{matrix}
1/2 & 0 & 0 & 1/2\\
1/2 & 1 & 1 & 1/2\\
1/2 & 0 & 0 & 1/2\\
1/2 & 1 & 1 & 1/2
\end{matrix}$ &
$\begin{matrix}
-1 & 1 & -4&0\\
0 & -1 & 1&0\\
0 & -1 & 1&0\\
-1 & -1 & -2&0
\end{matrix}$ &
$\begin{matrix}
 -1/8 & -3/8 & -1/8 & -3/8\\
 1/8 & -1/8 & 1/8 & -1/8\\
 -1/8 & 1/8 & -1/8 & 1/8\\
 1/4 & -1/4 & 1/4 & -1/4
\end{matrix}$ \\ \hline

$j=18$ &
$\begin{matrix}
-1 & -2 & -1 & 0\\
-1 & -2 & -1 & 0\\
-1 & -2 & -1 & 0\\
-1 & -2 & -1 & 0
\end{matrix}$ &
$\begin{matrix}
0 & 1/2 & -1/2&0\\
0 & -1/2 & 1/2&0\\
0 & -1/2 & 1/2&0\\
0 & 1/2 & -1/2&0
\end{matrix}$ &
$\begin{matrix}
 0 & 0 & 1/2 & 1/2\\
 0 & 0 & 0 & 0\\
 0 & 0 & -1/2 & -1/2\\
 0 & 0 & 1/2 & 1/2
\end{matrix}$ \\ \hline

$j=19$ &
$\begin{matrix}
-1/2 & -1 & -1 & 1/2\\
-1/2 & -1 & -1 & 1/2\\
-1/2 & -1 & 0 & -1/2\\
-1/2 & -1 & 0 & -1/2
\end{matrix}$ &
$\begin{matrix}
0 & 1/2 & -1/2& 0\\
0 & -1/2 & -1/2&-1/2\\
0 & -1/2 & 1/2&0\\
0 & 1/2 & -1/2&0
\end{matrix}$ &
$\begin{matrix}
 -1/4 & -1/4 & -5/4 & -5/4\\
 1/4 & 1/4 & 1/4 & 1/4\\
 -1/4 & -1/4 & 3/4 & 3/4\\
 1/2 & 1/2 & -1/2 & -1/2
\end{matrix}$ \\ \hline

$j=20$ &
$\begin{matrix}
0 & 0 & -1 & -1\\
0 & 0 & -1 & -1\\
0 & 0 & -1 & -1\\
0 & 0 & -1 & -1
\end{matrix}$ &
$\begin{matrix}
0 & -3/2 & 3/2&0\\
0 & 1/2 & -1/2&0\\
0 & 1/2 & -1/2&0\\
0 & 1/2 & -1/2&0
\end{matrix}$ &
$\begin{matrix}
 1/4 & 1/4 & 0 & 0\\
 -1/4 & -1/4 & 0 & 0\\
1/4 & 1/4 & 0 & 0\\
 0 & 0 & 0 & 0
\end{matrix}$ \\ \hline

$j=21$ &
$\begin{matrix}
0 & -1 & 1 & 0\\
0 & -1 & 1 & 0\\
0 & 1 & -1 & 0\\
0 & 1 & -1 & 0
\end{matrix}$ &
$\begin{matrix}
0 & -1/2 & -1/2&-1/2\\
0 & -1/2 & -1/2&-1/2\\
0 & 1/2 & 1/2&1/2\\
0 & 1/2 & 1/2&1/2
\end{matrix}$ &
$\begin{matrix}
 0 & -1/2 & 0 & 1/2\\
 0 & 1/2 & 0 & -1/2\\
 0 & 0 & 0 & 0\\
 0 & 0 & 0 & 0
\end{matrix}$ \\ \hline

$j=22$ &
$\begin{matrix}
-1/2 & 0 & -1 & -1/2\\
1/2 & 0 & 1 & 1/2\\
1/2 & 0 & 0 & -1/2\\
-1/2 & 0 & 0 & 1/2
\end{matrix}$ &
$\begin{matrix}
-1 & -1 & -4&-1\\
1 & 1 & 2&0\\
-1 & -1 & 0&1\\
-1 & -1 & 0&1
\end{matrix}$ &
$\begin{matrix}
 1/8 & -1/8 & -1/8 & 1/8\\
 -1/8 & 1/8 & 1/8 & -1/8\\
 1/8 & -1/8 & -1/8 & 1/8\\
 -1/4 & 1/4 & -1/4 & 1/4
\end{matrix}$ \\ \hline

$j=23$ &
$\begin{matrix}
-1 & -1 & -1 & 1\\
1 & 1 & 1 & -1\\
1 & 1 & 1 & -1\\
1 & 1 & 1 & -1
\end{matrix}$ &
$\begin{matrix}
0 & 1/2 & -1/2&-1/2\\
0 & -1/2 & 1/2&1/2\\
0 & -1/2 & 1/2&1/2\\
0 & 1/2 & -1/2&-1/2
\end{matrix}$ &
$\begin{matrix}
 0 & 0 & 1/4 & 1/4\\
 0 & 0 & -1/4 & -1/4\\
 0 & 0 & -1/4 & -1/4\\
 0 & 0 & 1/2 & 1/2
\end{matrix}$ \\ \hline

$j=24$ &
$\begin{matrix}
0 & 0 & 1 & 1\\
0 & 0 & -1 & -1\\
0 & 0 & -1 & -1\\
0 & 0 & -1 & -1
\end{matrix}$ &
$\begin{matrix}
0 & -1/2 & 1/2&1/2\\
0 & 1/2 & -1/2&-1/2\\
0 & 1/2 & -1/2&-1/2\\
0 & 1/2 & -1/2&-1/2
\end{matrix}$ &
$\begin{matrix}
 -1/4 & 0 & 1/4 & 0\\
 1/4 & 0 & -1/4 & 0\\
 -1/4 & 0 & 1/4 & 0\\
0 & 0 & 0 & 0
\end{matrix}$ \\ \hline

$j=25$ &
$\begin{matrix}
-1 & 0 & 0 & -1\\
1 & 0 & 2 & -1\\
1 & 0 & 0 & 1\\
1 & 0 & 2 & -1
\end{matrix}$ &
$\begin{matrix}
0 & 3/2 & 1/2&-1/2\\
0 & -1/2 & -1/2&0\\
0 & -1/2 & 1/2&1/2\\
0 & 1/2 & -1/2&-1/2
\end{matrix}$ &
$\begin{matrix}
 3/8 & -3/8 & -3/8 & -3/8\\
 -1/8 & 1/8 & 1/8 & 1/8\\
 -1/8 & 1/8 & 1/8 & 1/8\\
 1/4 & 1/4 & -1/4 & 1/4
\end{matrix}$ \\ \hline

$j=26$ &
$\begin{matrix}
0 & 0 & -1 & -1\\
0 & 0 & 1 & 1\\
0 & 0 & -1 & -1\\
0 & 0 & -1 & -1
\end{matrix}$ &
$\begin{matrix}
1/2 & 1/2 & 1&1/2\\
1/2 & 1/2 & 1&1/2\\
-1/2 & -1/2 & -1&-1/2\\
-1/2 & -1/2 & -1&-1/2
\end{matrix}$ &
$\begin{matrix}
 0 & -1/4 & 0 & 1/4\\
 0 & 1/4 & 0 & -1/4\\
 0 & -1/4 & 0 & 1/4\\
 0 & 0 & 0 & 0
\end{matrix}$ \\ \hline

$j=27$ &
$\begin{matrix}
0 & -1 & -1 & 0\\
0 & 1 & 1 & 0\\
0 & -1 & -1 & 0\\
0 & 1 & 1 & 0
\end{matrix}$ &
$\begin{matrix}
-3/2 & -3/2 & -3&0\\
1/2 & 1/2 & 1&0\\
1/2 & 1/2 & 1&0\\
-1/2 & -1/2 & -1&0
\end{matrix}$ &
$\begin{matrix}
 0 & 1/2 & 0 & 1/2\\
0 & 0 & 0 & 0\\
 0 & 0 & 0 & 0\\
 0 & 0 & 0 & 0
\end{matrix}$ \\ \hline

$j=28$ &
$\begin{matrix}
1/2 & 1 & 0 & -1/2\\
1/2 & 1 & 1 & 1/2\\
1/2 & 1 & 0 & -1/2\\
-1/2 & -1 & -1 & -1/2
\end{matrix}$ &
$\begin{matrix}
1/2 & 1/2 & 0&0\\
1/2 & 1/2 & 1&1/2\\
-1/2 & -1/2 & 0&0\\
-1/2 & -1/2 & 0&0
\end{matrix}$ &
$\begin{matrix}
 3/4 & -1/4 & 3/4 & 1/4\\
 -1/4 & -1/4 & -1/4 & 1/4\\
-1/4 & 3/4 & -1/4 & -3/4\\
 1/2 & -1/2 & 1/2 & 1/2
\end{matrix}$ \\ \hline

$j=29$ &
$\begin{matrix}
1 & 2 & 1 & 0\\
1 & 2 & 1 & 0\\
-1 & -2 & -1 & 0\\
1 & 2 & 1 & 0
\end{matrix}$ &
$\begin{matrix}
0 & 1/2 & 1/2&0\\
0 & -1/2 & -1/2&0\\
0 & -1/2 & -1/2&0\\
0 & -1/2 & -1/2&0
\end{matrix}$ &
$\begin{matrix}
 1/2 & 0 & -1/2 & 0\\
 0 & 0 & 0 & 0\\
 -1/2 & 0 & 1/2 & 0\\
 1/2 & 0 & -1/2 & 0
\end{matrix}$ \\ \hline

$j=30$ &
$\begin{matrix}
-1 & -1 & -1 & 1\\
1 & 1 & 1 & -1\\
-1 & -1 & -1 & 1\\
1 & 1 & 1 & -1
\end{matrix}$ &
$\begin{matrix}
1/2 & 1/2 & 1&0\\
-1/2 & -1/2 & -1&0\\
1/2 & 1/2 & 1&0\\
-1/2 & -1/2 & -1&0
\end{matrix}$ &
$\begin{matrix}
1/4 & 0 & 1/4 & 0\\
-1/4 & 0 & -1/4 & 0\\
 -1/4 & 0 & -1/4 & 0\\
1/2 & 0 & 1/2 & 0
\end{matrix}$ \\ \hline

$j=31$ &
$\begin{matrix}
0 & -1 & -1 & 0\\
0 & 1 & 1 & 0\\
0 & 1 & 1 & 0\\
0 & -1 & -1 & 0
\end{matrix}$ &
$\begin{matrix}
1/2 & 1/2 & 0&-1/2\\
-1/2 & -1/2 & 0&1/2\\
-1/2 & -1/2 & 0&1/2\\
-1/2 & -1/2 & 0&1/2
\end{matrix}$ &
$\begin{matrix}
 1/2 & 0 & -1/2 & 0\\
 0 & 0 & 0 & 0\\
 0 & 0 & 0 & 0\\
 0 & 0 & 0 & 0
\end{matrix}$ \\ \hline

$j=32$ &
$\begin{matrix}
1/2 & 0 & 1 & 1/2\\
1/2 & 0 & 1 & 1/2\\
-1/2 & 0 & 0 & 1/2\\
-1/2 & 0 & 0 & 1/2
\end{matrix}$ &
$\begin{matrix}
0 & 3 & -1&1\\
0 & -1 & 1&0\\
0 & -1 & -1&-1\\
0 & -1 & -1&-1
\end{matrix}$ &
$\begin{matrix}
 -1/8 & -1/8 & 1/8 & 1/8\\
 1/8 & 1/8 & -1/8 & -1/8\\
 -1/8 & -1/8 & 1/8 & 1/8\\
 -1/4 & -1/4 & -1/4 & -1/4
\end{matrix}$ \\ \hline

$j=33$ &
$\begin{matrix}
0 & 0 & -1 & 1\\
0 & 0 & -1 & 1\\
0 & 0 & 1 & -1\\
0 & 0 & -1 & 1
\end{matrix}$ &
$\begin{matrix}
0 & 3/2 & 3/2&0\\
0 & -1/2 & -1/2&0\\
0 & -1/2 & -1/2&0\\
0 & 1/2 & 1/2&0
\end{matrix}$ &
$\begin{matrix}
 0 & -3/4 & 0 & -3/4\\
 0 & 1/4 & 0 & 1/4\\
 0 & 1/4 & 0 & 1/4\\
 0 & 0 & 0 & 0
\end{matrix}$ \\ \hline

$j=34$ &
$\begin{matrix}
-1/2 & -1 & 0 & 1/2\\
1/2 & 0 & 1 & -1/2\\
1/2 & 1 & 0 & -1/2\\
-1/2 & 0 & -1 & 1/2
\end{matrix}$ &
$\begin{matrix}
0 & 1 & 1&1\\
-1 & -1 & 0&1\\
1 & 1 & 0&-1\\
0 & -1 & -1&-1
\end{matrix}$ &
$\begin{matrix}
 1/8 & -3/8 & -1/8 & 3/8\\
 -1/8 & 3/8 & 1/8 & -3/8\\
 1/8 & 1/8 & -1/8 & -1/8\\
 -1/4 & -1/4 & 1/4 & 1/4
\end{matrix}$ \\ \hline

$j=35$ &
$\begin{matrix}
1 & 0 & 0 & 1\\
-1 & 0 & -2 & 1\\
1 & 0 & 0 & 1\\
1 & 0 & 2 & -1
\end{matrix}$ &
$\begin{matrix}
1/2 & 1/2 & -1&-1/2\\
-1/2 & -1/2 & 0&0\\
1/2 & 1/2 & 1&1/2\\
-1/2 & -1/2 & -1&-1/2
\end{matrix}$ &
$\begin{matrix}
 3/8 & -3/8 & 3/8 & 3/8\\
 -1/8 & 1/8 & -1/8 & -1/8\\
 -1/8 & 1/8 & -1/8 & -1/8\\
-1/4 & -1/4 & -1/4 & 1/4
\end{matrix}$ \\ \hline

$j=36$ &
$\begin{matrix}
0 & 0 & 1 & 1\\
0 & 0 & -1 & -1\\
0 & 0 & 1 & 1\\
0 & 0 & -1 & -1
\end{matrix}$ &
$\begin{matrix}
-1/2 & -1/2 & -1&0\\
1/2 & 1/2 & 1&0\\
-1/2 & -1/2 & -1&0\\
-1/2 & -1/2 & -1&0
\end{matrix}$ &
$\begin{matrix}
 0 & 0 & -1/4 & 1/4\\
 0 & 0 & 1/4 & -1/4\\
 0 & 0 & -1/4 & 1/4\\
 0 & 0 & 0 & 0
\end{matrix}$ \\ \hline

$j=37$ &
$\begin{matrix}
0 & 1 & 1 & 0\\
0 & -1 & -1 & 0\\
0 & 1 & 1 & 0\\
0 & 1 & 1 & 0
\end{matrix}$ &
$\begin{matrix}
1/2 & 1/2 & 1&1/2\\
-1/2 & -1/2 & -1&-1/2\\
-1/2 & -1/2 & -1&-1/2\\
1/2 & 1/2 & 1&1/2
\end{matrix}$ &
$\begin{matrix}
 0 & 0 & -1/2 & -1/2\\
 0 & 0 & 0 & 0\\
 0 & 0 & 0 & 0\\
 0 & 0 & 0 & 0
\end{matrix}$ \\ \hline

$j=38$ &
$\begin{matrix}
1/2 & 1 & 1 & -1/2\\
-1/2 & -1 & -1 & 1/2\\
-1/2 & 0 & 0 & 1/2\\
-1/2 & 0 & 0 & 1/2
\end{matrix}$ &
$\begin{matrix}
1/2 & 3/2 & 0&-1/2\\
0 & -1/2 & 1/2&1/2\\
0 & -1/2 & 1/2&1/2\\
-1/2 & -1/2 & -1&-1/2
\end{matrix}$ &
$\begin{matrix}
 3/4 & -3/4 & -1/4 & -1/4\\
 -1/4 & 1/4 & -1/4 & -1/4\\
 -1/4 & 1/4 & -1/4 & -1/4\\
 1/2 & -1/2 & 1/2 & 1/2
\end{matrix}$ \\ \hline

$j=39$ &
$\begin{matrix}
0 & 1 & 1 & 0\\
0 & 1 & 1 & 0\\
0 & 1 & 1 & 0\\
0 & -1 & -1 & 0
\end{matrix}$ &
$\begin{matrix}
3/2 & 3/2 & 0&0\\
-1/2 & -1/2 & 0&0\\
-1/2 & -1/2 & 0&0\\
-1/2 & -1/2 & 0&0
\end{matrix}$ &
$\begin{matrix}
 -1/2 & -1/2 & 0 & 0\\
 0 & 0 & 0 & 0\\
 0 & 0 & 0 & 0\\
 0 & 0 & 0 & 0
\end{matrix}$ \\ \hline

$j=40$ &
$\begin{matrix}
-1/2 & 0 & 0 & -1/2\\
-1/2 & -1 & -1 & -1/2\\
1/2 & 0 & 0 & 1/2\\
1/2 & 1 & 1 & 1/2
\end{matrix}$ &
$\begin{matrix}
-1/2 & -3/2 & -1&-1/2\\
0 & 1/2 & 1/2&1/2\\
0 & 1/2 & 1/2&1/2\\
-1/2 & -1/2 & 0&1/2
\end{matrix}$ &
$\begin{matrix}
 3/4 & 1/4 & -3/4 & -1/4\\
 1/4 & -1/4 & -1/4 & 1/4\\
 -1/4 & 1/4 & 1/4 & -1/4\\
 1/2 & -1/2 & -1/2 & 1/2
\end{matrix}$ \\ \hline

$j=41$ &
$\begin{matrix}
1 & 2 & 1 & 0\\
-1 & -2 & -1 & 0\\
1 & 2 & 1 & 0\\
-1 & -2 & -1 & 0
\end{matrix}$ &
$\begin{matrix}
1/2 & 1/2 & 1&0\\
1/2 & 1/2 & 1&0\\
-1/2 & -1/2 & -1&0\\
1/2 & 1/2 & 1&0
\end{matrix}$ &
$\begin{matrix}
 -1/2 & 1/2 & 0 & 0\\
0 & 0 & 0 & 0\\
 1/2 & -1/2 & 0 & 0\\
-1/2 & 1/2 & 0 & 0
\end{matrix}$ \\ \hline

$j=42$ &
$\begin{matrix}
-1 & -1 & -1 & 1\\
1 & 1 & 1 & -1\\
1 & 1 & 1 & -1\\
-1 & -1 & -1 & 1
\end{matrix}$ &
$\begin{matrix}
-1/2 & -1/2 & 0&1/2\\
-1/2 & -1/2 & 0&1/2\\
1/2 & 1/2 & 0&-1/2\\
1/2 & 1/2 & 0&-1/2
\end{matrix}$ &
$\begin{matrix}
 0 & 1/4 & 0 & -1/4\\
 0 & -1/4 & 0 & 1/4\\
 0 & -1/4 & 0 & 1/4\\
 0 & 1/2 & 0 & -1/2
\end{matrix}$ \\ \hline

$j=43$ &
$\begin{matrix}
-1 & -1 & -1 & -1\\
-1 & -1 & -1 & -1\\
-1 & -1 & -1 & -1\\
1 & 1 & 1 & 1
\end{matrix}$ &
$\begin{matrix}
1/2 & 1/2 & 0 & 0\\
-1/2 & -1/2 & 0 & 0\\
1/2 & 1/2 & 0 & 0\\
1/2 & 1/2 & 0 & 0
\end{matrix}$ &
$\begin{matrix}
0 & 0 & -1/4& 1/4\\
0 & 0 & -1/4 & 1/4\\
0 & 0 & 1/4 & -1/4\\
0 & 0 & -1/2 & 1/2
\end{matrix}$ \\ \hline

$j=44$ &
$\begin{matrix}
-1/2 & 0 & -1 & -1/2\\
1/2 & 0 & 1 & 1/2\\
-1/2 & -1 & 0 & -1/2\\
-1/2 & -1 & 0 & -1/2
\end{matrix}$ &
$\begin{matrix}
0 & 1 & -1 & -1\\
-1 & -1 & -2 & -1\\
1 & 1 & 2 & 1\\
0 & 1 & -1 & -1
\end{matrix}$ &
$\begin{matrix}
 1/8 & -1/8 & 3/8 & 3/8\\
 -3/8 & 3/8 & -1/8 & -1/8\\
 1/8 & -1/8 & -1/8 & -1/8\\
 -1/4 & 1/4 & 1/4 & 1/4
\end{matrix}$ \\ \hline

$j=45$ &
$\begin{matrix}
0 & 0 & 1 & -1\\
0 & 0 & 1 & -1\\
0 & 0 & 1 & -1\\
0 & 0 & -1 & 1
\end{matrix}$ &
$\begin{matrix}
-1/2 & -1/2 & 0 & 0\\
1/2 & 1/2 & 0 & 0\\
-1/2 & -1/2 & 0 & 0\\
1/2 & 1/2 & 0 & 0
\end{matrix}$ &
$\begin{matrix}
 -3/4 & 0 & -3/4 & 0\\
 1/4 & 0 & 1/4 & 0\\
 1/4 & 0 & 1/4 & 0\\
 0 & 0 & 0 & 0
\end{matrix}$ \\ \hline

$j=46$ &
$\begin{matrix}
1 & 0 & 1 & 0\\
-1 & 0 & -1 & 0\\
1 & 0 & 1 & 0\\
1 & 0 & 1 & 0
\end{matrix}$ &
$\begin{matrix}
-1/2 & -1/2 & -1 & -1/2\\
1/2 & 1/2 & 1 & 1/2\\
-1/2 & -1/2 & -1 & -1/2\\
1/2 & 1/2 & 1 & 1/2
\end{matrix}$ &
$\begin{matrix}
 0 & 0 & 0 & 0\\
 0 & 0 & 0 & 0\\
 0 & 0 & 0 & 0\\
 -1/2 & 0 & -1/2 & 0
\end{matrix}$ \\ \hline

$j=47$ &
$\begin{matrix}
-1 & -1 & -1 & -1\\
-1 & -1 & -1 & -1\\
-1 & -1 & -1 & -1\\
-1 & -1 & -1 & -1
\end{matrix}$ &
$\begin{matrix}
0 & 3/2 & -3/2 & 0\\
0 & -1/2 & 1/2 & 0\\
0 & -1/2 & 1/2 & 0\\
0 & 1/2 & -1/2 & 0
\end{matrix}$ &
$\begin{matrix}
 0 & -1/4 & 0 & -1/4\\
 0 & -1/4 & 0 & -1/4\\
 0 & 1/4 & 0 & 1/4\\
 0 & -1/2 & 0 & -1/2
\end{matrix}$ \\ \hline

$j=48$ &
$\begin{matrix}
-1/2 & -1 & -1 & 1/2\\
-1/2 & -1 & -1 & 1/2\\
1/2 & 0 & 0 & -1/2\\
-1/2 & 0 & 0 & 1/2
\end{matrix}$ &
$\begin{matrix}
-1 & 1 & 2 & 0\\
0 & -1 & -1 & 0\\
0 & -1 & -1 & 0\\
1 & 1 & 0 & 0
\end{matrix}$ &
$\begin{matrix}
 1/8 & 1/8 & -3/8 & 3/8\\
 1/8 & 1/8 & 1/8 & -1/8\\
 1/8 & 1/8 & 1/8 & -1/8\\
 -1/4 & -1/4 & -1/4 & 1/4
\end{matrix}$ \\ \hline

\end{longtable}

\end{document}